\documentclass[reqno]{amsproc}
\usepackage{amssymb}
\usepackage{euscript} 
\usepackage{mathrsfs} 
\usepackage{color}

\makeatletter
\@namedef{subjclassname@2010}{%
\textup{2010} Mathematics Subject Classification}
\makeatother

\numberwithin{equation}{section}

\newtheorem{thm}{Theorem}[section]
\newtheorem{cor}[thm]{Corollary}
\newtheorem{lem}[thm]{Lemma}
\newtheorem{pro}[thm]{Proposition}

\newtheorem*{thm*}{Theorem}

\newtheorem*{opq*}{Problem}

\theoremstyle{remark}
\newtheorem{rem}[thm]{Remark}

\theoremstyle{definition}
\newtheorem{exa}[thm]{Example}
\DeclareMathOperator{\E}{e}

\newcommand*{\ascr}{\mathscr{A}}

\newcommand*{\cbb}{\mathbb C}

\newcommand*{\D}{\mathrm{d}}

\newcommand*{\dbbc}{\bar{\mathbb{D}}}

\newcommand*{\Ge}{\geqslant}

\newcommand*{\hh}{\mathcal H}

\newcommand*{\I}{{\mathrm i\hspace{.1ex}}}

\newcommand*{\is}[2]{\langle#1,#2\rangle}
\newcommand*{\jd}[1]{\mathscr N(#1)}

\newcommand*{\kk}{\mathcal K}
\newcommand*{\Le}{\leqslant}
\newcommand*{\mscr}{\mathscr M}
\newcommand*{\nbb}{\mathbb N}

\newcommand*{\nul}{\{0\}}
\newcommand*{\ogr}[1]{\boldsymbol B(#1)}

\newcommand*{\rbb}{\mathbb R}
\newcommand{\re}{{\rm Re\,}}

\newcommand*{\tbb}{{\mathbb T}}

\newcommand*{\zbb}{\mathbb Z}

\begin{document}
   \title[Criteria for algebraic operators to be unitary]{Criteria for algebraic operators to be unitary}
   \author[Z.\ J.\ Jab{\l}o\'nski]{Zenon Jan
Jab{\l}o\'nski}
   \address{Instytut Matematyki,
Uniwersytet Jagiello\'nski, ul.\ \L ojasiewicza 6,
PL-30348 Kra\-k\'ow, Poland}
\email{Zenon.Jablonski@im.uj.edu.pl}
   \author[I.\ B.\ Jung]{Il Bong Jung}
   \address{Department of Mathematics, Kyungpook National University,
Da\-egu 41566, Korea} \email{ibjung@knu.ac.kr}
   \author[J.\ Stochel]{Jan Stochel}
\address{Instytut Matematyki, Uniwersytet
Jagiello\'nski, ul.\ \L ojasiewicza 6, PL-30348
Kra\-k\'ow, Poland} \email{Jan.Stochel@im.uj.edu.pl}
   \thanks{The research of the first and third
authors was supported by the NCN
(National Science Center), decision No.\
DEC-2021/43/B/ST1/01651. The second
author was supported by Basic Science
Research Program through the National
Research Foundation of Korea (NRF) funded
by the Ministry of Education
(NRF-2021R111A1A01043569).}

   \maketitle
   \begin{abstract}
Criteria for an algebraic operator $T$ on
a complex Hilbert space $\hh$ to be
unitary are established. The main one is
written in terms of the convergence of
sequences of the form
$\{\|T^nh\|\}_{n=0}^{\infty}$ with $h\in
\hh$. Related questions are also
discussed.
   \end{abstract}
   \section{Introduction}
By the spectral theorem, a unitary
operator with a finite spectrum is
algebraic and its spectrum is contained
in $\tbb$, the unit circle centered at
$0$. The most fundamental example of a
unitary algebraic operator is the Fourier
transform. According to the famous
theorem of Plancherel, the Fourier
transform extends uniquely to a unitary
operator on $L^2(\rbb)$ (see e.g.,
\cite[Theorem~IX.6]{R-S}). Denote it by
$\mathscr F$. The Fourier transform
$\mathscr F$ has the following
properties:
   \begin{align*}
\text{$\mathscr F^0=I$, $\mathscr
F^1=\mathscr F$, $\mathscr F^2 = P$,
$\mathscr F^3 = \mathscr F^{-1}$ and
$\mathscr F^4=I$,}
   \end{align*}
where $I$ is the identity operator and
$P(f)(x) = f(-x)$. This implies that
$p(x)=x^4$ is the minimal polynomial of
$\mathscr F$. As a consequence, the
Fourier transform is a unitary algebraic
operator with (purely point) spectrum
$\sigma(\mathscr F)=\{1,-1,\I,-\I\}$ (see
\cite[Theorems~IX.1 and IX.6]{R-S}).

A natural question arises under what
additional assumptions an algebraic
(bounded linear) operator $T$ on a
complex Hilbert space $\hh$ with spectrum
in $\tbb$ is unitary. To answer this
question let us look at some broader
classes $\ascr$ of operators that can be
characterized as follows: an operator $T$
belongs to $\ascr$ if and only if the
sequences of the form $\{\|T^n
h\|^2\}_{n=0}^{\infty}$ ($h\in \hh$)
belong to the corresponding class
$\mathscr S$ of scalar sequences; in most
cases, the class $\mathscr S$ appears
naturally in harmonic analysis on
$*$-semigroups. In particular, the
celebrated theorem of Lambert states that
the class of subnormal operators
corresponds to Stieltjes moment sequences
(see \cite{lam}). In this line of
correspondence, we can list the classes
of $m$-isometric operators
\cite{Ag-St1,Ag-St2,Ag-St3,J-J-S20},
completely hypercontractive operators
\cite{Ag85}, completely hyperexpansive
operators \cite{At}, alternatingly
hyperexpansive operators \cite{Sh-At},
conditionally positive definite operators
\cite{J-J-S22}, and so on. The answer to
our question (see Theorem~\ref{algunit}
below) is written in terms of the
convergence of the sequences of the form
$\{\|T^n h\|\}_{n=0}^{\infty}$ ($h\in
\hh$). The condition of their convergence
seems to be optimal, since in the light
of Remark~\ref{lidw} the assumption of
their boundedness ceases to be
sufficient. It is also worth mentioning
that there are contractions (for which
the sequences $\{\|T^n
h\|\}_{n=0}^{\infty}$, $h\in \hh$,
automatically converge) with spectrum in
$\tbb$, called unimodular contractions,
which are not unitary (see \cite{Rus68}).
Clearly, unimodular contractions are
normaloid. Let us further note that if we
replace the class of normaloid operators
by a class of more regular operators, it
may turn out that members of the latter
class with spectrum in $\tbb$ are
unitary. In particular, by Stampfli's
theorem (see \cite[Corollary, p.\
473]{Sta65}), every hyponormal operator
with spectrum in $\tbb$ is unitary.

Before formulating the main result, we
establish some notation and terminology.
Denote by $\cbb$ the field of complex
numbers. Set $\tbb=\{z\in \cbb \colon
|z|=1\}$. Write $\nbb$, $\zbb_+$ and
$\rbb_+$ for the sets of positive
integers, nonnegative integers and
nonnegative real numbers, respectively.
Let $\ogr{\hh}$ stand for the
$C^*$-algebra of all bounded linear
operators on a complex Hilbert space
$\hh$. For $T\in \ogr{\hh}$, denote by
$\jd{T}$, $\sigma(T)$ and $r(T)$ the
kernel, the spectrum and the spectral
radius of $T$, respectively. An operator
$T \in \ogr{\hh}$ is said to be {\em
normaloid} if $r(T)=\|T\|$, or
equivalently, by Gelfand's formula for
spectral radius, if and only if
$\|T^n\|=\|T\|^n$ for all $n\in\nbb$.
Call $T \in \ogr{\hh}$ {\em algebraic} if
there exists a nonzero polynomial $p$ (in
one indeterminate with complex
coefficients) such that $p(T)=0$; such a
$p$ is said to be {\em minimal} if $p$ is
the (unique) monic polynomial of least
degree among all nonzero polynomials $q$
such that $q(T)=0$.

The following theorem, which is the main
result of the paper, characterizes
unitary algebraic operators in terms of
the convergence of the sequences of the
form $\{\|T^n h\|\}_{n=0}^{\infty}$. Its
proof is given in Section~\ref{Sec.3}.
   \begin{thm} \label{algunit}
Suppose that $T\in\ogr{\hh}$ is algebraic and
$\sigma(T) \subseteq \tbb$. Then the following
statements are equivalent{\em :}
   \begin{enumerate}
   \item[(i)] $T$ is unitary,
   \item[(ii)] $T$ is normaloid,
   \item[(iii)] $\|T\|\Le 1$ $($or equivalently, $\|T\|=1$$)$,
   \item[(iv)] the sequence $\{\|T^n h\|\}_{n=0}^{\infty}$
is convergent in $\rbb_+$ for every $h\in \hh$.
   \end{enumerate}
   \end{thm}
   \section{Preparatory facts} In this section we
give some basic facts about algebraic operators needed
in this paper. We begin with a purely linear algebra
result, the proof of which is left to the reader. If
$\mscr$ is a complex vector space, then the identity
transformation on $\mscr$ is denoted by $I_{\mscr}$
(or simply by $I$ if no ambiguity arises). We write
   \begin{align}  \label{sibs}
\mscr = \mscr_1 \dotplus \ldots \dotplus \mscr_m
   \end{align}
in the case when $\mscr$ is a direct sum of (finitely
many) vector subspaces $\mscr_1, \ldots, \mscr_m$.
   \begin{lem}
Suppose that \eqref{sibs} holds. Let $A\colon \mscr
\to \mscr$ be a linear transformation such that
$A(\mscr_j) \subseteq \mscr_j$ for all $j=1, \ldots,
m$ and let $z\in \cbb$. Then $A-z I_{\mscr}$ is a
bijection if and only if $A|_{\mscr_j}-z I_{\mscr_j}$
is a bijection for all $j=1, \ldots, m$. Moreover, if
$A-z I_{\mscr}$ is a bijection, then $(A-z
I_{\mscr})^{-1}(\mscr_j) = \mscr_j$ for all $j=1,
\ldots, m$ and
   \begin{align*}
(A|_{\mscr_j}-z I_{\mscr_j})^{-1} = (A-z
I_{\mscr})^{-1}|_{\mscr_j}, \quad j=1, \ldots, m.
   \end{align*}
   \end{lem}
   \begin{cor} \label{plalemc}
Suppose that $\hh$ is a complex Hilbert
space which is a direct sum of finitely
many nonzero closed vector subspaces
$\hh_1, \ldots, \hh_m$. Let
$T\in\ogr{\hh}$ be such that $T(\hh_j)
\subseteq \hh_j$ for all $j=1, \ldots,
m$. Then $\sigma(T|_{\hh_j}) \subseteq
\sigma(T)$ for all $j=1, \ldots, m$.
   \end{cor}
For the sake of self-containedness, we
sketch the proof of the following lemma
that collects indispensable facts about
algebraic operators.
   \begin{lem} \label{algop}
Let $T \in \ogr{\hh}$. Then the following conditions
are equivalent{\em:}
   \begin{enumerate}
   \item[(i)] $T$ is algebraic,
   \item[(ii)] there exist an integer $m\Ge 1$, integers
$i_1, \ldots, i_m \Ge 1$, distinct complex numbers
$z_1, \ldots, z_m$ and closed nonzero vector subspaces
$\hh_1, \ldots, \hh_m$ of $\hh$ such that
   \begin{enumerate}
   \item[(ii-a)] $\hh= \hh_1 \dotplus \ldots \dotplus
\hh_m$,
   \item[(ii-b)] $T(\hh_j) \subseteq \hh_j$ for all $j=1, \ldots,
m$,
   \item[(ii-c)] $(T_j -z_j I_j)^{i_j} = 0$ for all  $j=1, \ldots,
m$, where $T_j:=T|_{\hh_j}$ and~ $I_j:=I_{\hh_j}$,
   \item[(ii-d)] $\sigma(T)=\{z_1,\ldots,z_m\}$ and $\sigma(T_j) =
\{z_j\}$ for all $j=1, \ldots, m$,
   \item[(ii-e)] there exists a constant $c\in (0,\infty)$ such that
   \begin{align} \label{lobos}
\|h_j\| \Le c \|h_1 + \ldots + h_m\|, \quad j =1,
\ldots, m, \, h_1 \in \hh_1, \ldots, h_m \in \hh_m.
   \end{align}
   \end{enumerate}
   \end{enumerate}
   \end{lem}
   \begin{proof}
(i)$\Rightarrow$(ii) Let $T$ be an algebraic operator
and $p$ be its minimal polynomial. Clearly, $\deg p
\Ge 1$. It follows from the fundamental theorem of
algebra~that
   \begin{align*}
p(x)=(x-z_1)^{i_1}\cdots (x-z_m)^{i_m}
   \end{align*}
with unique integers $i_1, \ldots, i_m
\Ge 1$ and distinct complex numbers
$z_1,\ldots , z_m$. In view of
\cite[Lemma~ 6.1]{C-J-S09}, the condition
\mbox{(ii-a)} holds with $\hh_j:=\jd{(T
-z_j I)^{i_j}} \neq \nul$ for $j=1,
\dots, m$. This implies that $\hh_1,
\ldots, \hh_m$ are closed vector
subspaces of $\hh$ which are invariant
for $T$. As a consequence, \mbox{(ii-b)}
and \mbox{(ii-c)} hold. By the spectral
mapping theorem, $\sigma(T) \subseteq
\{z_1, \ldots, z_m\}$. Since $(T_j -z_j
I_j)^{i_j} = 0$ and $\hh_j \neq \nul$, we
infer from the spectral mapping theorem
and Corollary~ \ref{plalemc} that
   \begin{align*}
\{z_j\}=\sigma(T_j) \subseteq \sigma(T), \quad j=1,
\ldots, m,
   \end{align*}
which implies \mbox{(ii-d)}.

Now, we proceed to the proof of
\mbox{(ii-e)}. Define for $j=1,\ldots,m$
the linear projection $P_j\colon \hh \to
\hh$ by
   \begin{align*}
P_j(h_1 + \ldots + h_m)=h_j, \quad h_1 \in \hh_1,
\ldots, h_m \in \hh_m.
   \end{align*}
By \mbox{(ii-a)}, this definition is correct. Using
\cite[Lemma~ 6.1(iii)]{C-J-S09} we see that
   \begin{align*}
\kk_2:=\hh_2 \dotplus \ldots \dotplus \hh_m =
\mathscr{N}\bigg(\prod_{j=2}^m (T -z_j I)^{i_j}\bigg),
   \end{align*}
and so $\kk_2$ is a closed vector
subspace of $\hh$. Since, by
\mbox{(ii-a)}, $\hh=\hh_1\dotplus \kk_2$,
we infer from
\cite[Theorem~III.13.2]{con90} that
$P_1\in \ogr{\hh}$. A similar argument
shows that $P_j\in \ogr{\hh}$ for all
$j\in \{1,\ldots,m\}$. This implies
\eqref{lobos}.

(ii)$\Rightarrow$(i) It is enough to note that $p(T)=0$
with $p(x)=(x-z_1)^{i_1}\cdots (x-z_m)^{i_m}$. This
completes the proof.
   \end{proof}
   \section{\label{Sec.3}Proof of the main result}
   We begin by stating an auxiliary lemma. For the
reader's convenience, we give its proof, which is
essentially the same as that of Lemma~2.1 in the
unpublished paper \cite{Ja-Ju-St22}.
   \begin{lem} \label{rewb}
Let $b,w\in \cbb$ be such that $|w|=1$ and $w\neq \pm
1$. Assume that the sequence $\{\re(w^n
b)\}_{n=0}^{\infty}$ is convergent. Then $b=0$.
   \end{lem}
   \begin{proof}
Suppose to the contrary that $b\neq 0$.

Consider first the case when $w^m \neq 1$ for all
$m\in \nbb$. Then, by Jacobi's theorem (see
\cite[Theorem~ I.3.13]{Dev}), the sequence
$\{w^n\}_{n=0}^{\infty}$ is dense in $\tbb$. Hence,
there exist two subsequences
$\{w^{k_n}\}_{n=0}^{\infty}$ and
$\{w^{l_n}\}_{n=0}^{\infty}$ of the sequence
$\{w^n\}_{n=0}^{\infty}$ such that
   \begin{align*}
\text{$\lim_{n\to\infty} w^{k_n} = \frac{\bar b}{|b|}$
\quad and \quad $\lim_{n\to\infty} w^{l_n} = -
\frac{\bar b}{|b|}$.}
   \end{align*}
As a consequence, we have
   \begin{align*}
|b|=\lim_{n\to\infty} \re(w^{k_n} b) =
\lim_{n\to\infty} \re(w^{l_n} b) = -|b|,
   \end{align*}
which contradicts $b\neq 0$.

Suppose now that $w^m=1$ for some $m\in \nbb$. By our
assumption on $w$, $m$ must be greater than or equal
to $3$. It is easily seen that $\re(b), \re(wb),
\ldots, \re(w^{m-1}b)$ are the cluster points of the
sequence $\{\re(w^nb)\}_{n=0}^{\infty}$. Hence, we
have
   \begin{align} \label{rebe}
\re(b)=\re(w^jb), \quad j\in \nbb.
   \end{align}
Observe that
   \begin{align} \label{rebe2}
|b|=|w^jb|, \quad j\in \nbb.
   \end{align}
By \eqref{rebe} and \eqref{rebe2} with
$j=1$, either $b=wb$, which contradicts
$w\neq 1$, or $\bar b=wb$. Next, by
\eqref{rebe} and \eqref{rebe2} with
$j=2$, either $b=w^2b$, a contradiction,
or $\bar b = w^2 b$ which, together with
$\bar b=wb$, leads to a contradiction.
This completes the proof.
   \end{proof}
Before proving the main result of this paper, we
characterize power bounded algebraic operators with
spectrum in the unit circle. Recall that $T\in
\ogr{\hh}$ is said to be {\em power bounded} if
$\sup_{n\in\zbb_+} \|T^n\| < \infty$.
   \begin{lem} \label{algunit-a}
Let $T\in\ogr{\hh}$. Then the following
conditions are equivalent{\em :}
   \begin{enumerate}
   \item[(i)] $T$ is a power bounded algebraic
operator such that $\sigma(T) \subseteq
\tbb$,
   \item[(ii)] there exist an integer $m\Ge 1$,  closed
nonzero vector subspaces $\hh_1, \ldots,
\hh_m$ of $\hh$ and distinct complex
numbers $z_1, \ldots, z_m$ such that
   \begin{enumerate}
   \item[(ii-a)] $\hh= \hh_1 \dotplus
\ldots \dotplus \hh_m$,
   \item[(ii-b)] $T(h_1 + \ldots + h_m) = z_1 h_1 + \ldots + z_m
h_m$ for all $h_1 \in \hh_1, \ldots, h_m
\in \hh_m$,
   \item[(ii-c)] $\{z_1, \ldots, z_m\}
\subseteq \tbb$.
   \end{enumerate}
   \end{enumerate}
   \end{lem}
   \begin{proof}
(i)$\Rightarrow$(ii) Assume that $T$ is a
power bound algebraic operator such that
$\sigma(T) \subseteq \tbb$. By
Lemma~\ref{algop}, there exist an integer
$m\Ge 1$, integers $i_1, \ldots, i_m \Ge
1$, distinct complex numbers $z_1,
\ldots, z_m$ and closed nonzero vector
subspaces $\hh_1, \ldots, \hh_m$ of $\hh$
that satisfy the conditions
\mbox{(ii-a)}-\mbox{(ii-d)} of this
lemma. Since $\sigma(T) \subseteq \tbb$,
we have
   \begin{align} \label{zmod}
|z_l| = 1, \quad l=1, \ldots, m.
   \end{align}
Fix $k\in\{1, \ldots, m\}$. Take
$h_k\in\hh_k \setminus \nul$. Using
\eqref{zmod} ,
Lemma~\ref{algop}\mbox{(ii-c)} and
\cite[Sublemma 6.3]{C-J-S09}, we deduce
that
   \begin{align} \label{zbieg}
\text{$\alpha_{n}(h_k):=\frac{1}{n^{N(h_k)}}
\|T_k^n h_k\|$ converges to a positive
real number as $n\to \infty$,}
   \end{align}
where $N(h_k)$ is the unique nonnegative
integer such that
   \begin{align} \label{zbieg2}
\text{$(T_k-z_kI_k)^{N(h_k)}h_k \neq 0$
and $(T_k-z_kI_k)^{N(h_k)+1} h_k = 0$.}
   \end{align}
In particular, we have
   \begin{align}  \label{nier}
\|T_k^n h_k\| = n^{N(h_k)}
\alpha_{n}(h_k), \quad n \Ge 1.
   \end{align}
Since $T_k$ is power bounded, the
sequence $\{\|T_k^n
h_k\|\}_{n=0}^{\infty}$ is bounded.
Hence, one can infer from \eqref{zbieg}
and \eqref{nier} that $N(h_k)=0$. This,
together with \eqref{zbieg2}, implies
that $T_k h_k= z_k h_k$. As a
consequence, the system $T$, $z_1$,
\ldots, $z_m$, $\hh_1$, \ldots, $\hh_m$
satisfies the conditions (ii-a), (ii-b)
and (ii-c).

(ii)$\Rightarrow$(i) By (ii-a) and
(ii-b), $p(T)=0$ with $p(x)=(x-z_1)\cdots
(x-z_m)$, which means that $T$ is an
algebraic operator. According to the
spectral mapping theorem, $\sigma(T)
\subseteq \{z_1, \ldots, z_m\}$. In turn,
by (ii-b) and the assumption that each
$\hh_i$ is nonzero, we deduce that $z_1,
\ldots, z_m$ are eigenvalues of $T$.
Therefore, by (ii-c), we have
   \begin{align*}
\sigma(T) = \{z_1, \ldots, z_m\}
\subseteq \tbb.
   \end{align*}
Now, using Lemma~\ref{algop}(ii-e) (or
the uniform boundedness principle), we
deduce the power boundedness of $T$ from
(ii-a), (ii-b) and (ii-c). This completes
the proof.
   \end{proof}
   \begin{rem} \label{lidw}
Regarding Theorem~\ref{algunit}, we note
that in view of Lemma~\ref{algunit-a}
there exist algebraic operators $T$ with
spectrum in $\tbb$ that are not unitary
but have the property that each sequence
$\{\|T^n h\|\}_{n=0}^{\infty}$ is bounded
(or equivalently, $T$ is power bounded).
   \hfill $\diamondsuit$
   \end{rem}
   Now, we are ready to prove the main
result of this paper.
   \begin{proof}[Proof of Theorem~\ref{algunit}]
The implications
(i)$\Rightarrow$(ii)$\Rightarrow$(iii)$\Rightarrow$(iv)
are obvious.

(iv)$\Rightarrow$(i) Assume that $T$ is
algebraic, $\sigma(T) \subseteq \tbb$ and
the sequence $\{\|T^n
h\|\}_{n=0}^{\infty}$ is convergent in
$\rbb_+$ for every $h\in \hh$. By the
uniform boundedness principle, $T$ is
power bounded. In view of the implication
(i)$\Rightarrow$(ii) of
Lemma~\ref{algunit-a}, there exist an
integer $m\Ge 1$, distinct complex
numbers $z_1, \ldots, z_m$ and closed
nonzero vector subspaces $\hh_1, \ldots,
\hh_m$ of $\hh$ that satisfy the
conditions (ii-a), (ii-b) and (ii-c) of
Lemma~\ref{algunit-a}. Fix distinct $k,l
\in \{1, \ldots,m\}$. Take $h_k\in\hh_k$
and $h_l\in\hh_l$. Then, by the
conditions (ii-a), (ii-b) and (ii-c) of
Lemma~\ref{algunit-a}, we have
   \allowdisplaybreaks
   \begin{align*}
\|T^n (h_k+h_l)\|^{2}&=\|z_k^n h_k +
z_l^n h_l\|^{2}
   \\
&= \|h_k\|^2 + 2 \re \big((z_k\bar z_l)^n
\is{h_k}{h_l}\big) + \|h_l\|^2, \quad
n\Ge 0.
   \end{align*}
Combined with (iv), this implies that the
sequence $\big\{\re \big((z_k\bar z_l)^n
\is{h_k}{h_l}\big)\big\}_{n=0}^{\infty}$
is convergent. If $z_k\bar z_l=-1$, then
we see that $\re\is{h_k}{h_l}=0$.
Substituting $\I h_k$ in place of $h_k$,
we deduce that $\is{h_k}{h_l}=0$. The
only possibility left is that $z_k\bar
z_l\neq \pm 1$. Since, by the condition
(ii-c) of Lemma~\ref{algunit-a},
$|z_k\bar z_l|=1$, we infer from Lemma~
\ref{rewb} that $\is{h_k}{h_l}=0$. This
shows that $\hh= \hh_1 \oplus \cdots
\oplus \hh_m$. Hence, by the conditions
(ii-b) and (ii-c) of
Lemma~\ref{algunit-a}, we conclude that
$T$ is a unitary operator. This completes
the proof.
   \end{proof}
   \section{Related results}
For the reader's convenience, we record
here some useful facts related to the
main topic of this paper concerning
certain classes of operators. We begin by
discussing the question of the existence
of the limits $\lim_{n\to\infty} \|T^n
h\|$ in the case of normaloid operators
and the issue of strong
stability\footnote{We refer the reader to
\cite{Kub97,Kub11} for a discussion of
the different types of stability of
operators.} in the context of subnormal
operators. Both are intimately related to
Theorem~ \ref{algunit}.
   \begin{pro} \label{normaloid}
Let $T\in\ogr{\hh}$ be a normaloid operator. Then the
following conditions are equivalent{\em :}
   \begin{enumerate}
   \item[(i)] the sequence $\{\|T^n h\|\}_{n=0}^{\infty}$
is convergent in $\rbb_+$ for every $h\in \hh$,
   \item[(ii)] $T$ is power bounded,
   \item[(iii)] $T$ is a contraction.
   \end{enumerate}
   \end{pro}
   \begin{proof}
By the uniform boundedness principle and Gelfand's formula
for $r(T)$, (i) implies (ii) and (ii) implies (iii). That
(iii) implies (i) is obvious.
   \end{proof}
Before formulating the next result, we give the
necessary definitions and facts related to the concept
of subnormality. Recall that an operator $T\in
\ogr{\hh}$ is said to be {\em subnormal} if there
exist a complex Hilbert space $\kk$ and a normal
operator $N\in \ogr{\kk}$ such that $\hh \subseteq
\kk$ (isometric embedding) and $Th=Nh$ for all $h\in
\hh$. Such an $N$ is called a {\em normal extension}
of $T$; if $\kk$ has no proper closed vector subspace
containing $\hh$ and reducing $N$, then $N$ is called
{\em minimal}. By a {\em semispectral measure} of a
subnormal operator $T\in\ogr{\hh}$ we mean the Borel
$\ogr{\hh}$-valued measure $F$ on $\cbb$ defined by
   \begin{align}  \label{kompr}
F(\varDelta) = PE(\varDelta)|_{\hh}, \quad
\text{$\varDelta$ - Borel subset of $\cbb$},
   \end{align}
where $E$ is the spectral measure of a
minimal normal extension $N\in \ogr{\kk}$
of $T$ and $P\in \ogr{\kk}$ is the
orthogonal projection of $\kk$ onto
$\hh$. In view of \cite[Proposition~
5]{Ju-St}, a subnormal operator has
exactly one semispectral measure. We also
need the following fact (see
\cite[Proposition~4]{Ju-St}).
   \begin{align}  \label{ifjur}
   \begin{minipage}{55ex}
{\em If $N$ is minimal, then for every
Borel subset $\varDelta$ of $\cbb$,
\phantom{aaaaa} $F(\varDelta) = 0$ if and
only if $E(\varDelta) = 0$.}
   \end{minipage}
   \end{align}
According to \cite[Proposition~
II.4.6]{con91}), the following holds.
   \begin{align} \label{subn-norm}
\text{\em Any subnormal operator is normaloid.}
   \end{align}
We refer the reader to \cite{con91} for the
foundations of the theory of subnormal operators.
   \begin{pro}
Let $S\in\ogr{\hh}$ be a subnormal operator, $N\in
\ogr{\kk}$ be a minimal normal extension of $S$ and
$F$ be the semispectral measure of $S$. Then the
following assertions hold{\em :}
   \begin{enumerate}
   \item[(i)] if $S$ is a contraction, then
$\lim_{n\to\infty}\|S^n h\|^2 = \is{F(\tbb)h}{h}$ for
every $h\in \hh$,
   \item[(ii)] the following conditions are equivalent{\em :}
   \begin{enumerate}
   \item[(ii-a)] $S$ is strongly
stable, i.e., $\lim_{n\to\infty} S^n h = 0$ for every
$h\in \hh$,
   \item[(ii-b)] $S$ is power bounded and $F(\tbb)=0$,
   \item[(ii-c)] $S$ is a contraction and $F(\tbb)=0$,
   \end{enumerate}
   \item[(iii)] $S$ is
strongly stable if and only if $N$ is strongly stable.
   \end{enumerate}
   \end{pro}
   \begin{proof}
(i) Suppose $\|S\| \Le 1$. Then, by
\cite[Corollary~II.2.17]{con91},
$\|N\|=\|S\|\Le 1$. It follows from the
spectral theorem that
   \allowdisplaybreaks
   \begin{align*}
\|S^n h\|^2 & = \|N^n h\|^2
\overset{\eqref{kompr}}=\int_{\dbbc} |z|^{2n} \is{F(\D
z)h}{h}, \quad n\in \zbb_+, \, h\in \hh,
   \end{align*}
where $\dbbc=\{z\in \cbb \colon |z|\Le 1\}$. Since for
$z\in\dbbc$, the sequence
$\{|z|^{2n}\}_{n=0}^{\infty}$ converges to
$\chi_{\tbb}(z)$ as $n \to \infty$, where
$\chi_{\tbb}$ is the characteristic function of
$\tbb$, we deduce from Lebesgue's dominated
convergence theorem that (i) holds.

(ii) Since subnormal operators are
normaloid (see \eqref{subn-norm}), the
assertion (ii) follows from (i) and
Proposition~\ref{normaloid}.

(iii) Recall that $\|S\|=\|N\|$, and also
that $F(\tbb)=0$ if and only if
$E(\tbb)=0$, where $E$ is the spectral
measure of $N$ (see \eqref{ifjur}).
Combined with (ii) applied to both $S$
and $N$, this yields (iii).
   \end{proof}
We now estimate the growth of norms of powers of an
algebraic operator whose spectral radius is less than
or equal to $1$.
   \begin{pro}
Suppose that $T\in \ogr{\hh}$ is an algebraic operator
such that $0< r(T) \le 1$. Then there exists
$\alpha\in (0,\infty)$ such that
   \begin{align} \label{estym}
\|T^n\| \Le \alpha\, n^{\kappa} r(T)^n, \quad n\Ge 1,
   \end{align}
where $\kappa=(\deg p)-1$ and $p$ is the minimal
polynomial of $T$.
   \end{pro}
   \begin{proof}
Since $r(T)>0$, in light of Lemma~ \ref{algop} and its
proof it suffices to consider the case when
$(T-zI)^i=0$ for some $1 \Le i \Le \deg p$ and for
some $z\in \cbb$ such that $0 < |z| \Le 1$. Setting
$N=T-zI$, it is easily seen that
   \begin{align*}
T^n = (zI + N)^n=\sum_{j=0}^{i-1} \binom n j z^{n-j}
N^{j}, \quad n \Ge i-1,
   \end{align*}
which together with $r(T)=|z|$ implies that
   \begin{align*}
\|T^n\| \Le \sum_{j=0}^{i-1} \binom n j \|N^{j}\|
|z|^{n-j} \Le \bigg(\sum_{j=0}^{i-1} \frac{\|N^j\|}{j!
|z|^j}\bigg) n^{i-1} r(T)^n, \quad n \Ge i-1.
   \end{align*}
This completes the proof.
   \end{proof}
   \begin{rem}
a) It is worth mentioning that if $T\in \ogr{\hh}$ is
an algebraic operator such that $r(T)=0$, then, in
view of Lemma~\ref{algop}, the estimate \eqref{estym}
still holds, but only for $n\Ge \deg p$.

b) It follows from Gelfand's formula for
spectral radius that if $T\in \ogr{\hh}$
is such that $r(T) < 1$, then $T$ is
uniformly stable\footnote{A more detailed
discussion of this issue can be found in
\cite[Proposition~6.22]{Kub11}.}, i.e.,
$\lim_{n\to\infty} \|T^n\| = 0$, and
hence the sequence $\{T^n
h\}_{n=0}^{\infty}$ is convergent for
every $h\in\hh$. The latter statement
ceases to be true if the spectrum of $T$
has a nonempty intersection with $\tbb
\setminus \{1\}$, even if $T$ is
algebraic. It could be even worse,
namely, if $T=zI$, where $z=\E^{2\pi\I
\theta}$ and $\theta$ is an irrational
number, then by Jacobi's theorem (see
\cite[Theorem~I.3.13]{Dev}) the closure
of the set $\{T^n h\colon n \in \zbb_+\}$
is equal to $\{z h\colon z\in \tbb\}$ for
every $h\in \hh$.

c) Note that if $T\in \ogr{\hh}$ is algebraic, then
the sequence $\{\|T^n h\|^{1/n} \}_{n=1}^{\infty}$ is
convergent in $\rbb_+$ for all $h \in \hh$ (see
\cite[Proposition~ 6.2]{C-J-S09}). Without the
assumption that $T$ is algebraic, the sequence
$\{\|T^n h\|^{1/n}\}_{n=1}^{\infty}$ may not converge.
For more details on this issue, we refer the reader to
\cite{Dan}.
   \hfill $\diamondsuit$
   \end{rem}
We conclude this section by providing an example
illustrating Theorem~\ref{algunit} and
Proposition~\ref{normaloid}.
   \begin{exa}
Let $N\in \ogr{\hh}$ be a nonzero
operator such that $N^2=0$ and let
$\alpha\in \tbb$. Then the operator
$T_{\alpha}:=\alpha I + N$ is algebraic
because $p(T_{\alpha})=0$ with
$p(x)=(x-\alpha)^2$. Hence, by the
spectral mapping theorem
$\sigma(T_{\alpha}) = \{\alpha\}$ and so
$T_{\alpha}$ is invertible in $\ogr{\hh}$
and $r(T_{\alpha})=1$. However
$\|T_{\alpha}\|
> 1$ whenever $\|N\| > 2$ (this can be achieved simply
by rescaling $N$). As a consequence, $T_{\alpha}$ is
not normaloid, hence not subnormal (see
\eqref{subn-norm}). Observe that $\hh\setminus \jd{N}
\neq \emptyset$ and
   \begin{align} \label{jadro}
\text{$\lim_{n\to \infty} \|T_{\alpha}^n h\| = \infty$
if and only if $h \in \hh\setminus \jd{N}$.}
   \end{align}
Indeed, by Newton's binomial formula (or
simply by induction), we have
   \begin{align*}
(I+\bar \alpha N)^n = I + n \bar \alpha N, \quad n\in
\zbb_+,
   \end{align*}
which implies that
   \begin{align*}
\|(\alpha I+N)^n h\|^2 = \|h\|^2 + 2 n {\mathrm Re}
(\alpha \is{h}{Nh}) + n^2\|Nh\|^2, \quad n\in \zbb_+,
\, h\in \hh.
   \end{align*}
This yields \eqref{jadro}. Since $\hh\setminus \jd{N}
\neq \emptyset$, we infer from \eqref{jadro} that
$T_{\alpha}$ is not power bounded. A simple example of
this kind is $\hh=\cbb^2$,
$N=[\begin{smallmatrix} 0 & 1 \\
0 & 0 \end{smallmatrix}]$ and $\alpha=1$. In this
particular case $\|T_{1}\| > 1$ without rescaling.
   \hfill $\diamondsuit$
   \end{exa}
   \bibliographystyle{amsalpha}

\begin{thebibliography}{99}
   \bibitem{Ag85} J. Agler,  Hypercontractions and subnormality,
{\em J. Operator Theory} {\bf 13} (1985), 203-217.
   \bibitem{Ag-St1} J. Agler, M. Stankus, $m$-isometric
transformations of Hilbert spaces, I, {\it Integr.
Equ. Oper. Theory} {\bf 21} (1995), 383-429.
   \bibitem{Ag-St2} J. Agler, M. Stankus, $m$-isometric
transformations of Hilbert spaces, II, {\it Integr.
Equ. Oper. Theory} {\bf 23} (1995), 1-48.
   \bibitem{Ag-St3} J. Agler, M. Stankus, $m$-isometric
transformations of Hilbert spaces, III, {\it Integr.
Equ. Oper. Theory} {\bf 24} (1996), 379-421.
   \bibitem{At} A. Athavale, On completely hyperexpansive operators,
{\em Proc. Amer. Math. Soc.} {\bf 124} (1996),
3745-3752.
   \bibitem{C-J-S09} D. Cicho\'{n}, I.  B. Jung,
J. Stochel, Normality via local spectral
radii, {\em J. Operator Theory} {\bf 61}
(2009), 253-278.
   \bibitem{con90} J.  B.  Conway, {\it A course in functional
analysis}, Graduate Texts in Mathematics {\bf 96},
Springer-Verlag, New York, 1990.
   \bibitem{con91} J. B. Conway, {\em The theory of
subnormal operators}, Mathematical Surveys and
Monographs, {\bf 36}, American Mathematical Society,
Providence, RI, 1991.
   \bibitem{Dan} J. Dane\v{s}, On local spectral radius, {\em {\v
C}asopis P{\v e}st. Mat.} {\bf 112} (1987), 177-187.
   \bibitem{Dev} R. L. Devaney, {\em An
introduction to chaotic dynamical systems}, Reprint of
the second (1989) edition, Studies in Nonlinearity,
Westview Press, Boulder, CO, 2003.
   \bibitem{Ja-Ju-St22} Z. Jab{\l}o\'{n}ski, I. B.
Jung, J. Stochel, On the structure of
conditionally positive definite algebraic
operators, submitted.
   \bibitem{Ju-St} I. Jung, J. Stochel, Subnormal
operators whose adjoints have rich point
spectrum, {\em J. Funct. Anal.} {\bf 255}
(2008), 1797-1816.
   \bibitem{J-J-S20} Z. J. Jab{\l}o\'nski, I. B. Jung, J. Stochel,
$m$-Isometric operators and their local properties,
{\em Linear Algebra Appl.} {\bf 596} (2020), 49-70.
   \bibitem{J-J-S22} Z. J.
Jab{\l}o\'nski, I. B. Jung, J. Stochel,
Conditional positive definiteness in
operator theory, {\em Dissertationes
Math.} {\bf 578}, 2022, pp.\ 64.
   \bibitem{Kub97} C. S. Kubrusly,  {\em An introduction to models and
decompositions in operator theory}, Birkh\"{a}user Boston,
Inc., Boston, MA, 1997.
   \bibitem{Kub11} C. S. Kubrusly, {\em The elements of operator
theory}, Second edition,
Birkh\"{a}user/Springer, New York, 2011.
   \bibitem{lam} A. Lambert, Subnormality and weighted
shifts, {\em J. London Math. Soc.} {\bf 14} (1976),
476-480.
   \bibitem{R-S} M. Reed, B. Simon, {\em Methods of modern
mathematical physics, vol. I: Functional analysis,}
Academic Press, 1980.
   \bibitem{Rus68} B. Russo, Unimodular contractions
in Hilbert space, {\em Pacific J. Math.}
{\bf 26} (1968), 163-169.
   \bibitem{Sh-At} V. M. Sholapurkar,  A. Athavale,
Completely and alternatingly hyperexpansive operators,
{\em J. Operator Theory} {\bf 43} (2000), 43-68.
   \bibitem{Sta65} J. G. Stampfli, Hyponormal operators and spectral
density, {\em Trans. Amer. Math. Soc.} {\bf 117}
(1965), 469-476.
   \end{thebibliography}
   
   \end{document}